\documentclass[14pt,a4paper]{article}
\usepackage{amsthm, amsfonts, amssymb, amsmath,graphicx} 
\usepackage[french, ngerman, italian, english]{babel}
\usepackage[colorlinks=true]{hyperref}
\newtheorem{theorem}{Theorem}[section]
\newtheorem{lemma}[theorem]{Lemma}
\newtheorem{corollary}[theorem]{Corollary}
\newtheorem{proposition}[theorem]{Proposition}
\newtheorem{remark}[theorem]{Remark}
\theoremstyle{definition}
\newtheorem{definition}[theorem]{Definition}
\newtheorem{example}[theorem]{Example}
\newtheorem{remark/example}[theorem]{Remark/Example}

\newtheorem{algorithm}[theorem]{Algorithm} 
\numberwithin{equation}{section}
\newcommand{\Rees}{\operatorname{Rees}}
\newcommand{\ini}{\operatorname{in}}

\newcommand{\multiset}{\operatorname{multiset}}

\begin{document}

\title{The determinantal ideals of extended Hankel matrices}
\author{Le Dinh Nam\\
\footnotesize Dipartimento di Matematica\\
\footnotesize Universit\`a degli Studi di Genova, Italy\\
\footnotesize \url{ledinh@dima.unige.it}}
\maketitle
\vspace{15pt}
\begin{abstract}
\noindent
In this paper, we use the tools of Gr\"{o}bner bases and combinatorial secant varieties to study the determinantal ideals $I_t$ of the extended Hankel matrices. Denote by $c$-chain a sequence $a_1,\dots,a_k$ with $a_i+c<a_{i+1}$ for all $i=1,\dots,k-1$. Using the results of $c$-chain, we solve the membership problem for the symbolic powers $I_t^{(s)}$  and we compute the primary decomposition of the product $I_{t_1}\cdots I_{t_k}$ of the determinantal ideals. Passing through the initial ideals and algebras we prove that the product $I_{t_1}\cdots I_{t_k}$ has a linear resolution and the multi-homogeneous Rees algebra $\Rees(I_{t_1},\dots,I_{t_k})$ is defined by  a  Gr\"obner basis of quadrics.

\end{abstract}

\section{Introduction}

   The study of determinantal ideals, rings and varieties is a classical topic in commutative algebra, algebraic geometry and invariant  theory. The case of generic matrices is well understood, see  the book of Bruns and Vetter \cite{BV}, as well as that of  generic symmetric and generic skew-symmetric matrices, see the papers of J\'{o}zefiak \cite{J} and J\'{o}zefiak-Pragacz \cite{JP}. One of the possible ways to study these objects is via deformation to monomial ideals and this can be done by means of Gr\"obner bases. For the generic families, the Gr\"obner bases  have been described by  Sturmfels \cite{St1}, Herzog-Trung \cite{HT}, Conca \cite{C1}, Sturmfels-Sullivant \cite{SS} and Sullivant \cite{S}. The case of minors of Hankel matrices has been studied by Conca \cite{C}. In this paper, we deal with extended Hankel matrices. \\
Let $R=\mathbb{K}[x_1,x_2,\dots,x_n]$ where $\mathbb{K}$ is a field. Our goal is the study of the  ideal $I_t$ generated by the set of $t$-minors of  the matrix:   
$$
X_t= \left (
\begin{array}{ccccc}
x_1    & x_2    &  x_3   & \cdots & x_{n-(t-1)c} \\
x_{1+c}    & x_{2+c}    & \cdots & \cdots & \cdots    \\
x_{1+2c}    & \cdots & \cdots & \cdots & \cdots    \\ 
\vdots & \vdots & \vdots & \vdots & \vdots    \\ 
x_{1+(t-1)c}    & \cdots & \cdots & \cdots & x_n
\end{array}
\right ). 
$$

As we will explain, $I_2$  defines the (unique) balanced rational normal scroll of $\mathbb{P}^{n-1}$ of dimension $c$ and $I_t$ defines its $(t-1)$th secant variety.  Our goal is to study the symbolic powers and the primary decomposition of the powers of the ideals $I_t$ and the associated blow-up algebras. We will employ various techniques including those used in \cite{CHV,C,CH,SS,S,SU}.

We now describe our results in detail. We obtain the following descriptions of  the symbolic powers  and of the primary decomposition of the powers of $I_t$: 

$$I_t^{(s)}=\sum I_t^{a_t}I_{t+1}^{a_{t+1}}\cdots I_m^{a_m},$$
where $m=\lfloor \frac{n+c}{c+1} \rfloor$ and the sum  is extended  over all the sequences of non-negative integers
$a_t,$ $a_{t+1},$ $\dots,$ $a_m,$ with $a_t+2a_{t+1}+\cdots+(m-t+1)a_m=s$. The primary decomposition of $I_t^s$ is:
$$I_t^{s}=I_t^{(s)}\cap I_{t-1}^{(2s)} \cap \cdots \cap I_{1}^{(ts)}.$$ 
Both the description of the symbolic powers and of the primary decomposition are the expected ones in view of the results of De Concini, Eisenbud, Procesi \cite{DEP} and Sullivant \cite{S}.
 
Furthermore we also describe a primary decomposition of every product $I_{t_1}\cdots I_{t_s}$  and prove that $I_{t_1}\cdots I_{t_s}$ has a linear resolution. 
We show that the symbolic and the ordinary Rees algebras of $I_t$ are Cohen-Macaulay.  We also show that the Rees algebra of $I_t$ is   Koszul (indeed defined by a Gr\"obner basis of quadrics). Finally, we generalize these results, showing that the multi-Rees algebra  $\Rees(I_{t_1},\dots,I_{t_k})$ is Cohen-Macaulay and Koszul for every choice of the numbers $t_1,\dots,t_k$. 

Some results in this paper have been conjectured and confirmed by using the computer algebra package  CoCoA  \cite{CT}. This paper was made possible with the invaluable support from Prof. Aldo Conca. Many thanks also to Alexandru Constantinescu for his support.
 \section{Notation}
In this section, we recall some results of Simis-Ulrich \cite{SU} and Sturmfels-Sullivant \cite{SS,S} relating ideals defining secant varieties to their symbolic powers.\\
Let $I_1,\dots,I_r$ be ideals in a polynomial ring $R=\mathbb{K}[x]=\mathbb{K}[x_1,\dots,x_n]$ over a field $\mathbb{K}.$ The \textit{join} $I_1*\cdots*I_r$ is computed as follows. We use $rn$ new indeterminates, grouped into $r$ vectors $Y_j=(y_{j1},\dots,y_{jn}),$ $j=1,\dots,r$ and we consider the polynomial ring $\mathbb{K}[x,y]$  in all $rn+n$ indeterminates. Let $I_j(Y_j)$ be the image of the ideal $I_j$ in $\mathbb{K}[x,y]$ under the map $x\to y_j$. Then $I_1*I_2*\cdots*I_r$ is the elimination ideal
$$\big(I_1(y_1)+\cdots+I_r(y_r)+\big<y_{1i}+y_{2i}+\cdots+y_{ri}-x_i : i=1,\dots,n\big>\big)\cap \mathbb{K}[x].$$
We define the \textit{rth secant ideal} of an ideal $I\subset \mathbb{K}[x]$ to be the $r$-fold join of $I$ with itself: $$I^{\{r\}}:=I*I*\cdots*I.$$
If $\mathbb{K}=\bar{\mathbb{K}}$, $I$ and $J$ are homogeneous radical ideals defining varieties $V=V(I)$ and $W=V(J)$ then $I*J$ is the vanishing ideal of the embedded join $$V*W=\overline{\cup_{v\in V}\cup_{w\in W}\big<v,w\big>},$$
where $\big<v,w\big>$ is the line spanned by $v$ and $w$ and the closure operation is the Zariski closure. The join operation is commutative and associative. Moreover, it satisfies the distributive law with respect to intersection; see \cite[Lemma 2.1]{SS}.

Given a term order $\prec$  and an ideal $I$ of $\mathbb{K}[x]$ we denote by $\ini_{\prec}(I)$ the initial ideal of $I$ with respect to $\prec$. The authors proved in \cite[Theorem 2.3]{SU} that:
\begin{theorem} 
\label{thm1}
We have the following inclusion:
$$\ini_{\prec}(I_1*I_2*\cdots*I_r)\subseteq \ini_{\prec}(I_1)*\ini_{\prec}(I_2)*\cdots*\ini_{\prec}(I_r).$$
In particular, we have $$\ini_{\prec}(I^{\{r\}})\subseteq \big(\ini_{\prec}(I)\big)^{\{r\}}.$$ 
\end{theorem} 

\begin{definition} The term order $\prec$ is called \textit{delightful} for the ideal $I$ if $$\ini_{\prec}(I^{\{r\}})=\big(\ini_{\prec}(I)\big)^{\{r\}}$$ for all $r\geq1.$
\end{definition}
Let $G$ be an undirected graph with vertex set $[n]=\{1,2,\dots,n\}$. To $G$ we associate the \textit{edge ideal} $I(G)$ which is generated by the square-free quadratic monomials $x_ix_j$ corresponding to the edges $\{i,j\}$ of $G$.

The \textit{chromatic number} $\chi(G)$ of graph $G$ is the smallest number of colors which can be used to give a coloring of the vertices of $G$ such that no two adjacent vertices have the same color. The \textit{clique number} is the size of the largest complete subgraph. To the subset $V\subset [n]$ we associate the monomial $m_V=\prod_{i\in V}x_i$ and $G_V$ is the subgraph of $G$ associated with $V$. A graph $G$ is called \textit{perfect} if the chromatic number $\chi(G_V)$ equals the clique number $\omega(G_V)$ for every subset $V\subset [n]$.  In \cite[Theorem 3.2]{SS} and \cite[Proposition 3.4]{SS}, we have two following results:
\begin{theorem}
\label{thm2} 
The rth secant of an edge ideal $I(G)$ is generated by the square-free monomials $m_V$ whose subgraph $G_V$ is not r-colorable, that is:
$$I(G)^{\{r\}}=\big<m_V | \chi(G_V)>r\big>.$$
The minimal generators of $I(G)^{\{r\}}$ are those monomials $m_V$ such that $G_V$ is not r-colorable but $G_U$ is r-colorable for every proper subset $U\subset V.$
\end{theorem}
\begin{proposition} 
\label{prop3}
A graph $G$ is perfect if and only if the ideal $I(G)^{\{r\}}$ is generated in degree $r+1$ for every $r\in \mathbb{N}$ such that $I(G)^{\{r\}}\not=0.$

\end{proposition}

Let $I$ be a radical ideal in a polynomial ring $R$ over an algebraically closed field. We define the \textit{sth symbolic power} of $I$ to be  $$I^{(s)}=\bigcap_{p\in V(I)}m^s_p.$$
When $I$ is a prime ideal we known that $I^{(s)}$ is also the  $I$-primary component of $I^s$. In other words,  
$$I^{(s)}=R\cap I^sR_I=\{x\in R: \mbox{ there exists } f\in
R\setminus I  \mbox{ such that } fx\in I^s\}.$$ 
In \cite[Proposition 2.8]{S}, the author gives a formula to compute the symbolic power by join operation, namely:
$$I^{(r)}=I*m^r,$$
where $m=(x_1,\dots,x_n)$ is the maximal homogeneous ideal of $R$.\\
In characteristic zero, the symbolic power can also be computed by taking derivatives:
$$I^{(s)}=\Big<f \ \ |\frac{\partial^{|a|}f}{\partial x^a}\in I \textrm{ for all } a\in \mathbb{N}^n \textrm{ with } |a|=\sum_{i=1}^na_i\leq s-1\Big>.$$
Thus, the symbolic power $I^{(s)}$ contains all polynomials that vanish to order $s$ on the affine variety $V(I)$, and hence contains important geometric information about the variety.\\
Let I be a homogeneous radical ideal such that it does not containing linear forms. In \cite[lemma 2.5]{S}, one has $I^{(r)}\subseteq m^{r+1}.$ This implies that $$I^{\{r+s-1\}}=I^{\{r\}}*I^{\{s-1\}}\subseteq I^{\{r\}}*m^s=(I^{\{r\}})^{(s)}.$$
Moreover, the symbolic powers form a filtration $ (I^{\{r\}})^{(i)}(I^{\{r\}})^{(s-i)}\subseteq (I^{\{r\}})^{(s)}$ for all $i=1,\dots,s$. Hence, 
 $$\big(I^{\{r\}}\big)^{(s)}\subseteq I^{\{r+s-1\}}+\sum_{i=1}^{s-1}\big(I^{\{r\}}\big)^{(i)}\big(I^{\{r\}}\big)^{(s-i)}.$$
For many interesting families of ideals, the containment is an equality. This suggests the following definition.
\begin{definition} An ideal $I$ is \textit{differentially perfect} if for all $s$ and $r$, one has $$\big(I^{\{r\}}\big)^{(s)}=I^{\{r+s-1\}}+\sum_{i=1}^{s-1}\big(I^{\{r\}}\big)^{(i)}\big(I^{\{r\}}\big)^{(s-i)}.$$
\end{definition}
Note that an equivalent definition of \textit{r}-differentially perfect is that the symbolic powers of the secant ideal $I^{\{r\}}$ satisfy:
$$\big(I^{\{r\}}\big)^{(s)}=\sum_{\lambda \vdash s}\prod_j I^{\{r+\lambda_j-1\}},$$
where the sum runs over all partitions $\lambda=(\lambda_1,\lambda_2,\dots)$ of $s$, with $\lambda_i>0$. So if the ideal is differentially perfect then we can compute its symbolic powers in terms of secant ideals. One has \cite[Theorem 5.3]{S}: 
\begin{theorem} 
\label{thm5}
Let $I$ be an ideal and $\prec$ be a term order. Assume that $\prec$ is delightful for $I$ and $\ini_{\prec}(I)$ is radical and differentially perfect. Then $I$ is differentially perfect. In particular, let $\mathcal{G}_r=\{g_{r,1},g_{r,2},\dots\}$ be a $Gr\ddot{o}bner$ basis of $I^{\{r\}}$ with respect to $\prec$. Then the set of polynomials $$\mathcal{G}_{r,s}=\Big\{\prod_{i=1}^lg_{r_i,j_i}\ | r_i\geq r , \sum_{i=1}^l(r-r_i+1)=s\Big\}$$ is a $Gr\ddot{o}bner$ basis of $\big(I^{\{r\}}\big)^{(s)}$ with respect to $\prec$.
\end{theorem}

\section{The determinantal ideal of a extended Hankel matrix}
Denote by $<$ the degree lexicographic monomial order on $\mathbb{K}[x]$ induced by the order of the indeterminates $x_1>x_2>\dots >x_n$. In this section, we only use this term order. Fix an integer $c\geq 1$. Denote by $X$ the arrangement of indeterminates  

$$
X=
\begin{array}{cccccccccc}
x_1    & x_2    &  x_3   & \cdots & \cdots & \cdots & \cdots & x_{n-c}  & \cdots & x_n   \\
x_{1+c}    & x_{2+c}    & \cdots & \cdots & \cdots & \cdots & \cdots & x_n     \\  
x_{1+2c}    & \cdots & \cdots & \cdots & \cdots & x_n    \\
\vdots & \vdots & \vdots & \cdots \\
\vdots & \vdots & \vdots \\
x_{1+kc} & \cdots &x_n   
\end{array}
$$
where $k=\lfloor \frac{n-1}{c} \rfloor$. For all $j=1,\dots,k$ we denote  by $X_j$ the submatrix of $X$: 
$$
X_j=
\left (
\begin{array}{ccccc}
x_1    & x_2    &  x_3   & \cdots & x_{n-(j-1)c} \\
x_{1+c}    & x_{2+c}    & \cdots & \cdots & \cdots    \\
x_{1+2c}    & \cdots & \cdots & \cdots & \cdots    \\ 
\vdots & \vdots & \vdots & \vdots & \vdots    \\ 
x_{1+(j-1)c}    & \cdots & \cdots & \cdots & x_n
\end{array}
\right ).
$$

In particular, we have:
$$
X_2=
\left (
\begin{array}{ccccc}
x_1    & x_2    &  x_3   & \cdots & x_{n-c} \\ 
x_{1+c}    & \cdots & \cdots & \cdots & x_n
\end{array}
\right )
.$$\\
Given a matrix $A$ with entries in a ring we denote by $I_t(A)$ the ideal generated by all $t$-minors of the matrix $A$. It is well known that $I_2(X_2)$ is the defining ideal of the balanced rational normal scroll of dimension $c$ in $\mathbb{P}^{n-1}$, see \cite{CHV,H,E2}. For instance, let $n=7$ and $c=2$, permuting the columns of $X_2$ we may write it as
$$
\left (
\begin{array}{cccccc}
x_1    & x_3    &  x_5&  \big|  & x_2 & x_4 \\ 
x_3    & x_5    &  x_7 & \big| & x_4 & x_6
\end{array}
\right )
$$
and hence $I_2(X_2)$ defines the balanced scroll of dimension 2 in $\mathbb{P}^6,$ which is $S_{2,3}$ in the notation of \cite[page 93]{H}.

The highest order of a minor in $X$ is $\lfloor \frac{n+c}{c+1} \rfloor.$ Thus we consider only $t$-minors with $1\leq t\leq \lfloor \frac{n+c}{c+1} \rfloor$.  We set $m=\lfloor \frac{n+c}{c+1} \rfloor $ and denote by $I_t$ the determinantal ideal $I_t(X_t)$ for all $t=1,\dots,m.$

Given positive integers $a_1,a_2,\dots,a_s,b_1,b_2,\dots,b_s,$ with $a_i+b_j-1+(i-1)c\leq n$ for all $1\leq i,j\leq s,$ we denote by $[a_1,a_2,\dots,a_s|b_1,b_2,\dots,b_s]$ the minor of $X$ with row indices $a_1,a_2,\dots,a_s$ and column indices $b_1,b_2,\dots,b_s$. A minor of the form  $[1,2,\dots,s|b_1,b_2,\dots,b_s]$ will be called {\em maximal minor} or maximal $s$-minor. Note that each maximal minor is uniquely determined by the entries on the main diagonal. We denote by $M(a_1,\dots,a_s)$ the maximal $s$-minor, whose entries on the main diagonal are $x_{a_1},\dots,x_{a_s}$. It is easy to see that
$$M(a_1,a_2,\dots,a_s)=[1,2,\dots,s|a_1,a_2-c-1,\dots,a_s-(s-1)(c+1)].$$

For $t=1,\dots,\min(j+1,n-jc)$ let $I_t(X_j)$ be the ideal of $\mathbb{K}[x]$ generated by the $t$-minors of $X_j.$

Note first  that one has the following elementary relations:  
$$[a_1+1,\dots,a_t+1|b_1,\dots,b_t]=[a_1,\dots,a_t|b_1+c,\dots,b_t+c].$$ 
If $H\subseteq \{1,\dots,t\}$ we set  $e(H) =(e_1,\dots,e_t)$  where $e_i=1$ if $i\in H$ and $e_i=0$ if $i\not\in H$.
One has:
\begin{lemma}
\label{lm3} Let 
$\alpha=\alpha_1,\dots,\alpha_t$ and $\beta=\beta_1,\dots,\beta_t$ be sequences
of positive integers. Then for all $k=1,\dots,t$ one has  
\[
\sum_{H\subset \{1,\dots,t\},\ |H|=k} [\alpha+e(H)|\beta]=
\sum_{G\subset \{1,\dots,t\},\ |G|=k}[\alpha|\beta+c.e(G)]
\]
\noindent 
\end{lemma}

\begin{proof} (a) Set  $(-1)^H=(-1)^{\sum_{i\in H}i}$ and $\textbf{1}=(1,\dots,1)\in \mathbb{Z}^k$. Expanding
the minor $[\alpha+e(H)|\beta]$ with respect to the rows with indices by $H$ and
expanding the minor $[\alpha|\beta+ce(G)]$ with respect to the columns with
indices by $G$ one has: 
$$
\begin{array}{l}
\sum_{H} [\alpha+e(H)|\beta]=
\sum_{H}\sum_{G} (-1)^H(-1)^G [\alpha_H+\textbf{1}|\beta_G]
[\alpha^{\wedge H}|\beta^{\wedge G}]=\\ \\
\sum_{G}\sum_{H} (-1)^G(-1)^H [\alpha_H|\beta_G+c.\textbf{1}]
[\alpha^{\wedge H}|\beta^{\wedge G}]=
\sum_{G}[\alpha|\beta+c.e(G)].
\end{array}
$$
\end{proof}
\begin{corollary}
\label{co1lm3}
(a) If  $j>t$, then every $t$-minor of $X_j$ is a
linear combination of $t$-minors of $X_{j-1}.$\\
(b) Every $t$-minor of $X$ is a linear combination of maximal $t$-minors. In particular, if $A$ is a $t$-minor then we have $ A = A'+\sum_i\beta_iB_i$ with $A',B_i$ maximal $t$-minors, $\beta_i\in \mathbb{Z}$ and $ \ini(A)=\ini(A')>\ini(B_i),$ for all $i$.\\
(c) $I_t(X_{j+1})\subset I_t(X_j)$ for all $j=t,\dots,m-1.$
\end{corollary}
 
\begin{proof}(a) Let $[g|d]=[g_1,\dots,g_t|d_1,\dots,d_t]$ be a $t$-minor of $X_j$. Assume that $g_i<g_{i+1}$ and $d_i<d_{i+1}$ for all $i$. If $g_t<j$ then $[g|d]$ is already a $t$-minor of $X_{j-1}$. If $g_t=j$, then let $h$ be the smallest integer such that $g_h=j+h-t$. Now applying the equation in Lemma \ref{lm3} to the sequences $\alpha=g_1,\dots,g_{h-1},g_h-1,\dots,g_t-1$, $\beta=d$ and with $k=t-h+1$ one writes  $[g|d]$  as a linear combination  of $t$-minors which are either in $X_{j-1}$  or in $X_j$ but with a bigger $``h"$. Arguing by induction on $t-h$ one obtains the desired expression.

(b) and (c) follow immediately from (a) and one has $\ini(A)=\ini(A'),$ $ \ini(A) \neq \ini(B_{i}) , \ini(B_{i}) \neq \ini(B_{j})$ so we have $ \ini(A)=\ini(A')>\ini(B_{i}) $ for all $i$. 
\end{proof}
\begin{remark}
(a) If $c=1$, then we have $I_t(X_j)=I_t(X_t)$ for all $j=t+1,\dots,m$ (see \cite[Corollary 2.2]{C}).\\
(b) This assertion is not true in general for $c>1$. For example with  $c=2, n=8$ and $t=2$, we have $I_2(X_3)\not=I_2(X_2)$.  
\end{remark}

\begin{definition}
In $\mathbb{N}$  we introduce the following partial order:
 $$i\leq_c j \qquad \mbox{if and only if}\qquad i=j \mbox{ or }  i+c< j.$$
We write $i<_cj$ if $i\leq_c j$ and $i\not=j$. We say that a sequence of integers $a_1,a_2,\dots,a_s$ is a $c$-chain if
$a_1<_c a_2 <_c \cdots <_c a_s$. Similarly we say that a monomial $x_{a_1}\cdots x_{a_s}$ is a $c$-chain if its indices form a $c$-chain.
\end{definition}
We have a result relating $c$-chains and perfect graphs in the following way.
\begin{lemma}
\label{lm4}
Let $G$ be the graph  with vertices $V(G)=[n]$ and the set of edges $E(G)=\{(i,j) | i<_cj \; or\; j<_ci\}$. Then G is perfect.
\end{lemma}
\begin{proof}
Let $H$ be a subgraph of $G$.
Denote by $x_{a_1}x_{a_2}\cdots x_{a_k}$ the maximal $c$-chain with respect to term order $<$ which divides the monomial $\prod_{i\in H}x_i$. Obviously, the $c$-chain $a_1,\dots,a_k$ can be constructed as follows: 
\begin{itemize}
  \item [-] $a_1$ is the smallest vertex of $H$.
  \item [-] If the set $\{i|i\in H,a_{t-1}<_ci\}\not=\emptyset$, we set $a_t=\min\{i|i\in H,a_{t-1}<_ci\}$ for all $t\geq 2$.
\end{itemize}
 Recall $\chi(H)$ the chromatic number of the graph $H$ and $\omega(H)$ the clique number of the graph $H$. We will prove that $\chi(H)=\omega(H)=k.$ 

The subgraph of $H$ induced by the vertices $\{a_1,\dots,a_k\}$ is a $k$-complete subgraph of $H$. So $k\leq \omega(H).$
Assume that $\{b_1,\dots,b_l\}$ induces an $l$-clique of $H$. We have that $b=b_1,\dots,b_l$ is a $c$-chain of $H$. Because $\prod_{i=1}^lx_{b_i}\leq\prod_{j=1}^kx_{a_j}$, we get $l\leq k$. So $\omega(H)=k.$

Set $V_1=\{a_1,a_1+1,\dots,a_1+c\}$, $V_2=\{a_2,a_2+1,\dots,a_2+c\}$,\dots,$V_k=\{a_k,a_k+1,\dots,a_k+c\}$.
We have that $V_1\cap H$,\dots, $V_k\cap H$ is a $k$-coloring of $H$. So $\chi(H)\geq k$. Assume that $l=\chi(H)$ and $V_1,\dots,V_l$ are sets of colors. Denote  $g_t=\min\{j|j\in V_t\}$ for all $t=1,\dots,l$. If $g_1<g_2<\cdots<g_l$ then $g_1,\dots,g_l$ is a $c$-chain of $H$ so $l\leq k$, and thus $\chi(H)=k$.  
\end{proof}
The sum of $r$ matrices of rank $\leq$1 has rank $\leq r$. Hence the $(r+1)$-minors of $X$ vanish on $V(I_2^{\{r\}}).$ Now, the ideal $I_2$ is easily seen to be prime over any field, and hence $I_2^{\{r\}}$ is geometrically prime. This implies that $I_{r+1}\subseteq I_2^{\{r\}}$.

Using Buchberger's Algorithm, it is easy to prove the following lemma:
\begin{lemma}
\label{Gb1}
The set of 2-minors of $X_2$ is a Gr\"obner basis of $I_2.$ 
\end{lemma}
\begin{corollary}
\label{lmgb}With the above notation one has:$$\ini(I_2)=\big<x_{a_1}x_{a_2}:\;a_1,a_2\; is\;c-chain\; with \;length \;2\big>.$$
\end{corollary}
\begin{theorem}
\label{thm6}
The term order $<$ is delightful for $I_2.$
\end{theorem}
\begin{proof}
Let $G$ be the graph as in Lemma \ref{lm4}. We have $I(G)=\ini(I_2)$.
Because $G$ is a perfect graph, Theorem \ref{thm2} and Proposition \ref{prop3} imply that 
$$I(G)^{\{r\}}=\Big<m_V|\chi(V)>r\Big>=\Big<x_{a_0}x_{a_1}\cdots x_{a_r}|a_0,a_1,\dots,a_r \; is\; c-chain\Big>.$$
Each such monomial is the $<$-leading term of an $(r+1)$-minor of $X_{r+1}.$
This implies that $I(G)^{\{r\}}\subseteq \ini(I_{r+1})\subseteq\ini(I_2^{\{r\}})\subseteq (\ini(I_2))^{\{r\}}=I(G)^{\{r\}}.$ Hence, $\ini(I_2^{\{r\}})=(\ini(I_2))^{\{r\}}$ for all $r\geq 1$. 
\end{proof}
\begin{corollary}
\label{co1thm6}
The secant ideal $I_2^{\{r\}}$ is generated by the $(r+1)$-minors $$I_{r+1}=I_2^{\{r\}},$$
these minors form a $Gr\ddot{o}bner$ basis.
\end{corollary}
\begin{proof}
In the proof of Theorem \ref{thm6} we have argued that the $(r+1)$-minors lie in $I_2^{\{r\}}$, and their leading terms generate the initial ideal $\big(\ini(I_2)\big)^{\{r\}}=\big((I_2)^{\{r\}}\big).$ This implies that the $(r+1)$-minors form a $Gr\ddot{o}bner$ basis for the ideal $I_2^{\{r\}}.$ In particular, they generate that ideal.
\end{proof}

Let $k=k_1,k_2,\dots,k_s$ be a sequence of integers. We define the function
$$\gamma_t(k)=\sum_{i=1}^s \max\{ k_i+1-t, 0\}.$$

Let $\delta$ be a monomial of $R$. We now describe a canonical decomposition of $\delta$ into a product of  $c$-chains. First let  $\delta_1$ be the $c$-chain which divides $\delta$ and is maximal with respect to $<$. If $\delta_1\neq \delta$, then let $\delta_2$ be the $c$-chain which divides $\delta/\delta_1$ and is maximal with respect to $<$, and so on. We end up with a decomposition $\delta=\delta_1\delta_2\cdots \delta_k$ which is uniquely determined by $\delta$. It is called \textit{c-decomposition}. Denote by $s_i$ the degree of $\delta_i$. The sequence $s_\delta=s_1,s_2,\dots,s_k$ is called the shape of $\delta$. We define the function $ \gamma_{t,c}(\delta)=\gamma_t(s_\delta)$. One has:
\begin{lemma}
\label{cc}
Let $a$ and $b$ be two $c$-chains of length $s$, respectively $r$. Then the $c$-decomposition of $ab$ has at most two factors and one of them has length $\geq \max(s,r)$.
\end{lemma}
\begin{proof}
By the definition of $c$-decomposition, we only need to  show that the $c$-decomposition of $ab$ has at most two factors. Assume that $a=a_1,\dots,a_s$, $b=b_1,\dots,b_r$ are $c$-chains with $s\geq r$. We have 
$|\multiset\{a_1,\dots,a_s,b_1,\dots,b_r\}\cap [b_j,b_j+c]|\leq 2$ and $|\multiset\{a_1,\dots,a_s,b_1,\dots,b_r\}\cap [a_i,a_i+c]|\leq 2$ for all $i=1,\dots,s,j=1,\dots,r$.\\ If $(\alpha_1\cdots\alpha_t)(\beta_1\cdots\beta_p)(\gamma_1\cdots\gamma_q)\cdots$ is the $c$-decomposition of $ab$, we have $\alpha_i$,$\beta_j,$ $\gamma_k$ $\in \{a_1,\dots,a_s,b_1,\dots,b_r\}$ for all $i,j,k.$ Moreover, there exist $i_0,j_0$ such that $\gamma_1\in [\alpha_{i_0},\alpha_{i_0}+c]$ and $\gamma_1\in [\beta_{j_0},\beta_{j_0}+c]$. Assume that $\alpha_{i_0}\leq \beta_{j_0}$. This implies that $\alpha_{i_0},\beta_{j_0},\gamma_1\in [\alpha_{i_0},\alpha_{i_0}+c]$, a contraction.  
\end{proof}
\begin{corollary} Let $a$ and $b$ be two $c$-chains. Then $\gamma_{t,c}(ab)\geq  \gamma_{t,c}(a)+\gamma_{t,c}(b).$
\end{corollary}
 
We set $$J_t=\big<x_{a_1}\cdots x_{a_t}\; :\; a_1,a_2,\dots,a_t \mbox{ is a c-chain}\big>.$$
We have the following result:
\begin{theorem}
\label{lmsym}
The ideal $J_2$ is differentially perfect. In particular, the symbolic powers of the edge ideals $J_2^{\{r\}}$ are:
$$\big(J_2^{\{r\}}\big)^{(s)}=\big(J_{r+1}\big)^{(s)}=\Big<x^a |\gamma_{r+1,c}(x^a)\geq s\Big>.$$
\end{theorem}

We have that the ideal  $J_{r+1}$ is generated by all $c$-chains of length $r+1$ and hence it is a square-free monomial ideal
associated with a simplicial complex that we denote by $\Delta_r$. 
If $j=j_1,\dots,j_r$ is a $c$-chain with $j_r\leq n-c$ then the set
$F_j=\{j_1,j_1+1,\dots,j_1+c,\dots,j_r,j_r+1,\dots,j_r+c\}$ is clearly a facet of $\Delta_r$. Furthermore it is easy to see that any facets  of $\Delta_r$ is of the form $F_j$ for some $c$-chain  $j$ of length $r$ and bounded by $n-c$. Denote by $A_r$   the set of the $c$-chains of length $r$ bounded by $n-c$, and for $j\in A_r$ denote by $P_j$ the ideal $(x_i:i\not\in F_j)$.
We have:  
$$J_{r+1}=\bigcap_{j\in A_r}P_{j}.$$
So 
$\big(J_{r+1}\big)^{(s)}=\bigcap_{j\in A_r}P_{j}^s.$
To prove Theorem \ref{lmsym} we need the following results.
\begin{lemma}
\label{dtt}
Let $I_1,I_2,\dots,I_h,U_1,U_2,\dots,U_k$ $(h>k)$ be closed intervals of length $c$ in $\mathbb{R}$ such that 
$I_{\alpha}\cap I_{\beta}=\emptyset$, $U_{\alpha}\cap U_{\beta}=\emptyset$ ( for all $\alpha\not=\beta$)and $|\{i_1,i_2,\dots,i_h\}\bigcap (\bigcup_1^kU_t)|<k$ where $i_{\alpha}=\min(I_{\alpha})$. Then we can choose other disjoint closed intervals 
$U_1',\dots,U_k'$ in the set of closed intervals $\{I_1,I_2,\dots,I_h,U_1,$ $U_2,\dots,U_k\}$ such that $(\bigcup_1^h I_{\alpha})\bigcap (\bigcup_1^k U_{\beta})\subseteq (\bigcup_1^h I_{\alpha})\bigcap (\bigcup_1^k U_{\beta}')$ and \\
 $|\{i_1,i_2,\dots,i_h\}\cap (\bigcup_1^kU_t)| <|\{i_1,i_2,\dots,i_h\}\cap (\bigcup_1^kU_t')|.$
\end{lemma}
\begin{proof} Set $j_{\beta}=\min(U_{\beta})$ and $I_{h+1}=U_{k+1}=U_0=\emptyset$. We can assume that $i_t<_ci_{t+1}$ and $j_t<_cj_{t+1}$. We will prove the lemma by induction on $k$.\\

If $k=1$, we have 
\begin{equation}
\label{abc}
\{i_1,i_2,\dots,i_h\}\bigcap U_1=\emptyset.
\end{equation}
- If $(\bigcup_1^h I_{\alpha})\bigcap U_1=\emptyset $, we choose $U_1'=I_1$.\\
- If there exists $\alpha \in [h]$ such that $U\cap I_{\alpha}\not=\emptyset$, by (\ref{abc}) we have $U\cap I_{\alpha+1}=\emptyset$ and $i_{\alpha}<j_1$. Thus, we choose $U_1'=I_{\alpha}.$

Assume that the clause is true for all $i=1,\dots,k-1$. We have the following possible cases:\\
- There exists $\beta\in [k]$ such that $U_{\beta}\bigcap (\bigcup_1^h I_{\alpha})=\emptyset$ and 
   $$|\{i_1,i_2,\dots,i_h\}\bigcap (\bigcup_{t\not=\beta}U_t)|<k-1.$$ By induction, we can choose  $U_1',\dots,U_{\beta-1}',$ $U_{\beta+1}',\dots,U_k'$ in\\ $\{I_1,,\dots,I_h,U_1,\dots,U_{\beta-1},U_{\beta+1},\dots,U_k\}$  such that $$(\bigcup_1^h I_{\alpha})\bigcap (\bigcup_{t\not=\beta} U_t)\subseteq (\bigcup_1^h I_{\alpha})\bigcap (\bigcup_{t\not=\beta} U_t')$$ and $|\{i_1,i_2,\dots,i_h\}\cap (\bigcup_{t\not=\beta}U_t)| <|\{i_1,i_2,\dots,i_h\}\cap (\bigcup_{t\not=\beta} U_t')|.$ We have that  $U_1',\dots,U_{\beta-1}',U_{\beta},U_{\beta+1}',\dots,U_k'$ satisfies the condition.\\
- There exists $\beta\in [k]$ such that $U_{\beta}\bigcap (\bigcup_1^h I_{\alpha})=\emptyset$ and
   $|\{i_1,i_2,\dots,i_h\}\bigcap (\bigcup_{t\not=\beta}U_t)|$ $=k-1$. Assume that $i_{t_p}\in U_p$ for all $j=1,\dots,\beta-1,\beta+1,\dots,k$. So $t_p<t_{p+1}$ and $j_p\leq i_{t_p}.$ Because $h>k$, there exists an index $\alpha$ in set  $\{i_1,\dots,i_h\}-\{t_1,\dots,t_k\}$. If $I_{\alpha}\cap \bigcup_{t\not=\beta} U_t=\emptyset$, we choose $U_1,\dots,U_{\beta-1},U_{\beta+1},\dots,U_k, I_{\alpha}$. Otherwise, there exists a pair $(q,\epsilon)$ such that $I_{\alpha}\cap U_q\not=\emptyset,I_{\alpha+1}\cap U_{q+1}\not=\emptyset,\dots,I_{\alpha+\epsilon}\cap U_{q+\epsilon}=\emptyset$. We choose disjoint closed intervals $I_{\alpha},I_{\alpha+1},\dots,I_{\alpha+\epsilon}$ to replace $U_{\beta},U_q,U_{q+1},\dots,U_{q+\epsilon-1}$ in $U_1,U_2,\dots,U_k.$  
- We have $U_{\beta}\bigcap (\bigcup_1^h I_{\alpha})\not=\emptyset$ for all $\beta=1,\dots,k$. Because $h>k$ and $|\{i_1,i_2,\dots,i_h\}\bigcap (\bigcup_1^kU_t)|<k$, there exists a \textit{zigzag intersection}, i.e. there exists a triangle set $(\alpha,\epsilon,q)$ such that $I_{\alpha}\cap U_{q-1}=\emptyset$, $I_{\alpha}\cap U_q\not=\emptyset$, $I_{\alpha+1}\cap U_q\not=\emptyset$, \dots, $I_{\alpha+\epsilon}\cap U_{q+\epsilon}\not=\emptyset$ and $I_{\alpha+\epsilon+1}\cap U_{q+\epsilon}=\emptyset$. We choose disjoint closed intervals $I_{\alpha},I_{\alpha+1},\dots,I_{\alpha+\epsilon}$ to replace $U_q,U_{q+1},\dots,U_{q+\epsilon}$ in $U_1,U_2,\dots,U_k.$ 
\end{proof}
\begin{definition}
Let $\delta$ be a monomial and $P$ an ideal. We define the function $$O_P(\delta)=\max\{k: \delta\in P^k\}.$$
If monomial $\delta_1=x_{i_1}\cdots x_{i_s}$ is the maximum of the $c$-chains which divide $\delta$ then the $c$-chain $i=i_1,\dots,i_s$ is called the \textit{c-socle} of $\delta$, denoted by $Soc_c(\delta).$ We set $supp(\delta)=\{i:\; x_i|\delta\}$. By the maximality of $\delta_1$, we have $supp(\delta)\subseteq F_{Soc_c(\delta)}=\{i_1,i_1+1,\dots,i_1+c,i_2,i_2+1,\dots,i_2+c,\dots,i_s,i_s+1,\dots,i_s+c\}.$\\
Let $i=i_1,\dots,i_h$ and $j=j_1,\dots,j_k$ with $h>k$. We define $$r_i(j)=|\{i_1,\dots,i_h\}\cap F_j|.$$
\end{definition}
\begin{lemma}
\label{lmlmsym}
Let $\delta$ be a monomial in $\cap_{j\in A_r}P_{j}^s$. Then 
\begin{equation}
\label{pr1}
 O_{P_j}(\delta)+r_{Soc_c(\delta)}(j) \geq s+r 
\end{equation}

for all $j\in A_r.$
\end{lemma}
\begin{proof}
In this proof, we denote $r(j)=r_{Soc_c(\delta)}(j)$ for simplicity. We use decreasing induction on $r(j)$. In general we have
$r(j)\leq r$.  If $r(j)=r$ then (\ref{pr1}) is trivially true because $O_{P_j}(\delta)\geq s$ since $\delta$ is a monomial in $\cap_{j\in A_r}P_{j}^s$. Assume that $r(j)<r$. Set $G=\bigcup_{i=1}^sG_t$ where $ G_t=\{i_t,i_t+1,\dots,i_t+c\}$. By Lemma \ref{dtt}, there exists $z\in A_r$ such that $F_j\cap G\subset F_{z}$ and $r(z)>r(j)$. By straightforward computations we obtain that $O_{P_j}(\delta)+r(j)\geq O_{P_{z}}(\delta)+r(z)$. Moreover $r(z)>r(j)$, so (\ref{pr1}) follows by induction.  
\end{proof}
 
\begin{proof}[Proof of Theorem \ref{lmsym}] 
We need to prove that $$\bigcap_{j\in A_r}P_{j}^s=\Big<x^a |\gamma_{t,c}(x^a)\geq s\Big>.$$
Let $\delta$ be a monomial and $\delta_1\cdots \delta_p$ be a c-decomposition of $\delta$. Denote by $s_i$ the size of $\delta_i$. Each facet of $\Delta_r$ contains at most $r$ points of the support of $\delta_i$.  It follows that $\delta_i\in \cap_{j\in A_r}P_{j}^{\gamma_r(\delta_i)}$ and thus $\delta \in \cap_{j\in A_r}P_{j}^{\gamma_r(\delta)}$. By the definition of $\gamma_{t,c}(x^a)$ we have $\Big<x^a |\gamma_{t,c}(x^a)\geq s\Big>\subseteq \cap_{j\in A_r}P_{j}^s$. \\

Conversely, if $s_1\leq r$ then there exists $j\in A_r$ such that $supp(\delta)\subseteq \{i_1,i_1+1,\dots,i_1+c,\dots,i_s,i_s+1\dots,i_s+c\}\subset F_j$, and since $\delta\in P_j^s$, it follows that $s=0$ which is a trivial case. So we may assume that
$s_1\geq r+1$. Let $\eta=\delta/\delta_1$. We have $\gamma_{t,c}(\delta)=\gamma_{t,c}(\eta)+\gamma_{t,c}(\delta_1)=\gamma_{t,c}(\eta)+s_1-r$. 
By induction it suffices to show that $$\qquad \eta\in \cap_{j\in A_r}P_{j}^{s-s_1+r}$$
for all $j\in A_r$.\\
This means that: 
\begin{equation}
\label{pr2}
 O_{P_j}(\eta)\geq s-s_1+r.
\end{equation}
However, one has 
$$O_{P_j}(\eta)=O_{P_j}(\delta)-|\{h: i_h\not\in F_j\}|=O_{P_j}(\delta)-s_1+|\{h: i_h\in F_j\}|.$$
So (\ref{pr2}) is equivalent to   
$$    O_{P_j}(\delta)+r(j) \geq s+r $$
 for all $j\in A_r$, by Lemma \ref{lmlmsym}. 
\end{proof}
\begin{theorem}
\label{symbpow}
 For all $t=1,\dots,m$ and $s\in \mathbb{N}$ one has  
$$I_t^{(s)}=\sum I_t^{a_t}I_{t+1}^{a_{t+1}}\cdots I_m^{a_m}$$
  the sum  being extended  over all the sequences of non-negative integers
$a_t,$ $a_{t+1},$ $\dots,$ $a_m,$ with $a_t+2a_{t+1}+\cdots+(m-t+1)a_m=s$.
\end{theorem}
\begin{proof}
Set $r=t-1$. We have $in(I_t)=J_{r+1}$. By Corollary \ref{co1thm6} and Theorem \ref{lmsym}, the ideal $in(I_2^{\{r\}})$ is differentially perfect. However Theorem \ref{thm6} implies that any diagonal order $\prec$ is delightful for $I_2$. Thus, by Theorem \ref{thm5}, $I_2$ is differentially perfect. This means
$$\big(I_2^{\{r\}}\big)^{(s)}=I_2^{\{r+s-1\}}+\sum_{i=1}^{s-1}\big(I_2^{\{r\}}\big)^{(i)}\big(I_2^{\{r\}}\big)^{(s-i)}$$
or
$$\big(I_t\big)^{(s)}=I_{t+s-1}+\sum_{i=1}^{s-1}\big(I_t\big)^{(i)}\big(I_t\big)^{(s-i)}.$$
So by induction on $s$, we have
$$I_t^{(s)}=\sum I_t^{a_t}I_{t+1}^{a_{t+1}}\cdots I_m^{a_m},$$
  the sum  being extended  over all the sequences of non-negative integers
$a_t,$ $a_{t+1},$ $\dots,$ $a_m,$ with $a_t+2a_{t+1}+\cdots+(m-t+1)a_m=s$.

\end{proof}
\begin{corollary} If $g\in I_t^{(s)}$ then $\gamma_{t,c}(\ini(g))\geq s.$ 
\end{corollary}
\begin{proof} Since $g\in I_t^{(s)}$, we have $\ini(g)=\delta_1\delta_2\cdots\delta_p.g'$, where $\delta_i$ is a $c$-chain and $\sum_i\gamma_{t,c}(\delta_i)\geq s$. By Lemma \ref{cc}, we have $\gamma_{t,c}(\ini(g))\geq s.$ 
\end{proof}

We have  a bijective correspondence between the sets:
$$
\phi:\{ c-\mbox{chains of } \mathbb{K}[x]\} \to   \{\mbox{maximal
minors of } X\}
$$
defined by setting $\phi(x_{a_1}\cdots x_{a_s})=M(a_1,a_2,\dots,a_s)$. The inverse   of $\phi$ is  the map which takes every maximal minor to its initial monomial. Now $\phi$ induces a map 
 $$
\Phi:\{ \mbox{ordinary monomials of  } \mathbb{K}[x]\} \to   \{\mbox{products of maximal
minors of } X\},
$$ 
which is defined by $\Phi(\delta)=\phi(\delta_1)\phi(\delta_2)\cdots \phi(\delta_k)$ where 
$\delta=\delta_1\delta_2\cdots \delta_k$ is the $c$-decomposition of $\delta$. Note that by construction one has $\ini(\Phi(\delta))=\delta$ and hence $\Phi$ is injective. We now define the set of the {\em standard monomials} of $X$ to be the image of $\Phi$. So by construction we have a bijective correspondence:
$$
\Phi:\{ \mbox{ordinary monomials of  } \mathbb{K}[x]\} \to   \{\mbox{ standard monomials
of  } X\}
$$
whose inverse is given by the map which takes every standard monomial to its initial monomial, i.e. $\ini(\Phi(\delta))=\delta$ for all ordinary monomials $\delta$ and $ \Phi(\ini(\mu))=\mu$ for all standard monomials $\mu$. 
\begin{remark}
The standard monomials form a $\mathbb{K}$-basis of the polynomial ring $\mathbb{K}[x]$.   
\end{remark}
\begin{example}  
Let $c=2$ and $\delta=x_1^2x_2x_4x_7x_8x_{10}$. The $c$-decomposition
of  $\delta$ is $(x_1x_4x_7x_{10})(x_1x_8)(x_2)$ and the shape is $4,2,1$.
Thus  $\delta$ corresponds to the standard monomial 
$$\mu=M(1,4,7,10)M(1,8)M(2)=
\left(
\begin{array}{cccc}
x_1 & x_2 & x_3 & x_4\\
x_3 & x_4 & x_5 & x_6\\
x_5 & x_6 & x_7 & x_8\\
x_7 & x_8 & x_9 & x_{10}
\end{array}
\right)
\left(
\begin{array}{cc}
x_1 & x_6 \\
x_3 &  x_8\\
\end{array}
\right)
x_2.$$ 
In terms of tableau:
$$\delta=
\begin{array}{cccc}
1 & 4 & 7 & 10\\
1 & 8 \\
2 \\

\end{array}
$$
Obviously $\mu \in I_2^{(4)}\cap I_3^{(2)}$.

\end{example}
Given a product of minors $\delta$  of shape $s=s_1,\dots,s_k$ and $t\in \mathbb{N}$, one defines the function
$$\gamma_t(\delta)=\sum_{i=1}^k\max\{s_i-t+1,0\}.$$
Let $\delta=\delta_1\cdots\delta_u$ be a product of minors such that $\gamma_t(\delta)\geq s$. We can assume that $\delta\in I_1^{a_1}\cdots I_m^{a_m}$ ( $a_i\geq 0$ ). Since $\gamma_t(\delta)\geq s$ we have $a_t+2a_{t+1}+\cdots +(m-t+1)a_m\geq s$. Hence $\delta\in I_t^{(s)}$.\\
By Theorem \ref{lmsym} and Theorem \ref{symbpow} we have following corollaries:
\begin{corollary}
\label{corosymb}
Let $\Delta$ be a product of minors and $\Delta=\sum_{j=1}^p\lambda_j\Delta_j$ be a standard representation of $\Delta$. Then $\gamma_t(\Delta)\leq \gamma_t(\Delta_i)$ for all $t=1\dots,m$. 
\end{corollary}
\begin{proof}
Since $\Delta=\sum_{i=1}^p\lambda_i\Delta_i$ is the standard representation of $\Delta$, we can assume that $\ini(\Delta)=\ini(\Delta_1)>\ini(\Delta_2)>\cdots>\ini(\Delta_p).$ Set $h=\gamma_t(\Delta)$, $Q=\{i: i\in\{1,\dots,p\}\ and\ \gamma_t(\Delta_i)<h\}$ and $P=\{1,\dots,p\}-Q$. We need prove that $Q=\emptyset.$\\
If $Q\not=\emptyset$, we have $\gamma_t(\Delta_i)<h$ for all $i\in Q$. This implies $\gamma_{t,c}(\ini\Delta_i)<h$ for all $i\in Q$. Set $g=\sum_{i\in Q}\lambda_i\Delta_i=\Delta-\sum_{j\in P}\lambda_j\Delta_j$. So $g\in I_t^{(h)}$. Hence, $\gamma_{t,c}(\ini (g))\geq h$. But $\ini(g)=\ini(\Delta_{i_0})$ for some $i_0\in Q$. So we obtain a contraction.  
\end{proof}
We say that an ideal $I$ of $\mathbb{K}[x]$ is an {\em ideal of standard monomials} if $I$
has a basis as a $\mathbb{K}$-vector space which consists of standard monomials. 
The class of ideals of standard monomials is obviously closed under sum and
intersection and the fact that distinct standard monomials have distinct initial monomials. So if $I$ is an ideal of standard monomials and $B$ is a standard monomial $\mathbb{K}$-basis of $I$ then $B$ is a Gr\"obner basis of $I$ with respect to $<$. Furthermore, the monomials $\ini(\mu)$ with $\mu\in B$ form a $\mathbb{K}$-basis of $\ini(I)$. \\
Denote by $G_{t,s}$ the set  of the standard monomial $\mu$ which have  all the factors of
size  $\geq t$ and $\gamma_t(\mu)=s.$ By Theorem \ref{symbpow} and Corollary \ref{corosymb}, we have the following corollaries:
\begin{corollary}
\label{coro1}
The ideal $I_t^{(s)}$ is an ideal of standard monomials. In particular, the set of the standard monomials $\mu$ with $\gamma_t(\mu)\geq s$ is a $\mathbb{K}$-basis of $I_t^{(s)}$. Furthermore, $G_{t,s}$ is a Gr\"obner basis of $I_t^{(s)}.$
\end{corollary} 
\begin{corollary} 
\label{coro2}
The ideal $I_t$ has primary powers if and only if $t=1$ or $t=m$. 
\end{corollary}

For all the products of minors $\mu=\mu_1\cdots \mu_k$ of shape $\tau=t_1,t_2,\dots,t_k$ and  for all $j\in \mathbb{N}$ one has $\mu\in
I_j^{(\gamma_j(\tau))}$ and thus 
 $$I_{t_1}\cdots I_{t_k}\subseteq \bigcap_{j=1}^{t_1} I_j^{( \gamma_j(\tau))}.$$
Note that  $\cap_{j=1}^{t_1} I_j^{( \gamma_j(\tau))}$, being the intersection  of ideals of standard monomials, is an ideal of standard monomials.  Its $\mathbb{K}$-basis is the set of the standard monomials $\mu$ with $\gamma_j(\mu)\geq \gamma_j(\tau)$ for all $j=1,\dots,t_1$.  
\begin{lemma}
\label{prim7}
 Let $n_1$ and $n_2$ be $c$-chains of $\mathbb{K}[x]$ of length $s$ and $r,$ with $s>r+1.$ Then there exist two $c$-chains $n_3,n_4$ of length  $s-1$ and $r+1$ such that $n_1n_2=n_3n_4.$
\end{lemma}
\begin{proof} Let $n_1=x_{i_1}\cdots x_{i_s}$ and $n_2=x_{j_1}\cdots x_{j_r}$. For $h=1,\dots,r$ we set  $i_h^\prime=\min(i_h,j_h)$ and $j_h^\prime=\max(i_h,j_h)$. The sequences $i_1^\prime,\dots,i_r^\prime,i_{r+1},\dots,i_s$ and  $j_1^\prime,\dots,j_r^\prime$ are $c$-chains, and hence we may assume that $i_h\leq j_h$ for all $h=1,\dots,r$. We have to distinguish two cases:
\begin{itemize}
  \item[-] If $i_k<_c j_k$ for some $k$, we take  $k$ to be the minimum of the integers with this property. So $j_{k-1}\leq i_{k-1}+c<i_k<_c i_{k+1}$. Thus  $j_1,\dots,j_{k-1},i_{k+1},\dots,i_s$ and $i_1,\dots,i_k,j_k,\dots,j_r$ are $c$-chains and one takes $n_3$ and $n_4$ to be the associated monomials.
  \item[-] If $i_k\not <_c j_k$ for all $k$ then $j_r\leq i_r+c<i_{r+1}<_c i_s$. Thus  $i_1,\dots,i_{s-1}$ and  $j_1,\dots,j_{r},i_{s}$ are $c$-chains and one takes $n_3$ and $n_4$ to be the associated monomials.
\end{itemize}   
\end{proof}
\begin{lemma} 
\label{prim8}
Let $\tau=t_1,t_2,\dots,t_k$ be a sequence of integers with $m\geq t_1\geq t_2\geq \cdots\geq t_k\geq 1$.
Let $\mu=\mu_1\cdots \mu_q$ be a  product of minors such that $\gamma_j(\mu)\geq \gamma_j(\tau)$ 
for all $j=1,\dots,t_1$. Then there exists a product of minors $\delta_1,\dots,\delta_k$ of shape $\tau$ such that $\ini(\delta_1\cdots \delta_k)|\ini(\mu)$.  
\end{lemma}

\begin{proof} 
We use induction on $r=deg(\mu)$.

If one of the $\mu_i$ is a $t_1$-minor, we have $\gamma_j(\mu_1\cdots\mu_{i-1}\mu_{i+1}\cdots\mu_q)=\gamma_j(\mu)-(t_1+1-j)\geq
\gamma_j(\tau)-(t_1+1-j)=\gamma_j(t_2,\dots,t_k)$ for all $j=1,\dots,t_1$. By induction, there exists a product of minors $\delta'=\delta_1,\dots,\delta_{k-1}$ of shape $t_2,\dots,t_k$ such that $\ini(\delta_1\cdots \delta_{k-1})|\ini(\mu)$. So one has $\delta=\mu_i\delta'$.

Otherwise, we may arrange the factors $\mu_i$  in ascending order according to their size and assume that $\mu_1,\dots,\mu_p$ have size $<t_1$ and $\mu_{p+1},\dots,\mu_q$ have size $>t_1$. Let $r$ be the size of $\mu_p$ and $s$ be the size of $\mu_{p+1}$. By virtue of  Lemma \ref{prim7} we may find two minors $\rho_1$ and $\rho_2$ of size $r+1$ and $s-1$,  respectively, such that $\ini(\rho_1\rho_2)=\ini(\mu_p\mu_{p+1})$. Set $\mu^\prime=\mu_1\cdots\mu_{p-1}\rho_1\rho_2\mu_{p+2}\cdots\mu_q$. 
We note that $\gamma_j(\mu^\prime)\geq \gamma_j(\tau)$ for $j=1,\dots,t_1$.
Since $\ini(\mu)=\ini(\mu^\prime)$ and $\mu^\prime$ has either a factor of size $t_1$ or a smaller $``s-r"$, we may then conclude by induction.
. 
\end{proof}
We have the following theorem:
\begin{theorem}
\label{Primdec}
 Let $\tau=t_1,t_2,\dots,t_k$ be a sequence of integers with $m\geq t_1\geq t_2\geq \cdots\geq t_k\geq 1$. Then
$$I_{t_1}\cdots I_{t_k}= \bigcap_{j=1}^{t_1} I_j^{( \gamma_j(\tau))} $$
is a possibly redundant primary decomposition of $I_{t_1}\cdots I_{t_k}$. 
\end{theorem}
\begin{proof}
Let $J$ denote the ideal generated by the initial monomials of the products of minors of shape $\tau$. Since $\ini(\bigcap_{j=1}^{t_1} I_j^{( \gamma_j(\tau))})$ is generated by the initial monomials of the standard monomials $\mu$ with $\gamma_j(\mu)\geq
\gamma_j(\tau)$ for all $j=1,\dots,t_1$, by Lemma \ref{prim8} one has $\ini(\bigcap_{j=1}^{t_1} I_j^{( \gamma_j(\tau))})\subseteq J$. Since  $$I_{t_1}\cdots I_{t_k}\subseteq \bigcap_{j=1}^{t_1} I_j^{( \gamma_j(\tau))},$$
we have
$$
J\subseteq \ini(I_{t_1})\cdots \ini(I_{t_k})\subseteq  \ini(\bigcap_{j=1}^{t_1} I_j^{(
\gamma_j(\tau))})\subseteq J.
$$
It follows that 
$I_{t_1}\cdots I_{t_k}= \bigcap_{j=1}^{t_1} I_j^{( \gamma_j(\tau))}$.
\end{proof}
The proof of the theorem has the following important corollaries:
\begin{corollary} 
\label{prim9}
 Let $\tau=t_1,t_2,\dots,t_k$ be a sequence of integers with $m\geq t_1\geq t_2\geq \cdots\geq t_k\geq 1$. Then the product of minors of shape $\tau$ form a Gr\"obner basis of the ideal $I_{t_1}\cdots I_{t_k}$. In particular one has:
$$\ini(I_{t_1}\cdots I_{t_k})=\ini(I_{t_1})\cdots \ini(I_{t_k}).$$
\end{corollary}
\begin{corollary}
\label{prim10}
Let $\tau=t_1,t_2,\dots,t_k$ be a sequence of integers with $m\geq t_1\geq t_2\geq \cdots\geq t_k\geq 1$. Then
$$\ini(I_{t_1})\cdots \ini(I_{t_k})=\bigcap_{j=1}^{t_1} \bigcap_{z\in A_{j-1}} P_z^{\gamma_j(\tau)}$$
is a possibly redundant primary decomposition of $\ini(I_{t_1})\cdots \ini(I_{t_k})$. 
\end{corollary}
We can derive the following important results for the special case $t_1=\cdots=t_k=t,$ using the same arguments the author uses in \cite[Theorem 3.16]{C}.
\begin{theorem} 
\label{irredprimdec}
(a) Let $1\leq t\leq m$ and $k\in \mathbb{N}$. Set $u=\max(1,m-k(m-t))$. Then:
$$I_t^k=\bigcap_{j=u}^{t} I_j^{(k(t+1-j))}$$
is an irredundant  primary decomposition of $I_t^k$.\\
(b) $\ini(I_t^k)=\ini(I_t)^k$ for all $k$. 
\end{theorem}

\section{ Products of determinantal ideals with linear resolution}
In this section, we prove that any product of determinantal ideals
$$ I=I_{t_1}I_{t_2}\cdots I_{t_k}$$
has a linear resolution.

We know that the initial ideal of $I$ is $ J=J_{t_1}J_{t_2}\cdots J_{t_k}.$
We can assume that $m\geq t_1 \geq t_2 \geq \cdots \geq t_k\geq 1 .$

One says that an ideal $J\subseteq R=\mathbb{K}[x]$ has \textit{linear quotients} if $J$ has a system of generators $\mu_1,\dots,\mu_h$ such that for every $k=1,\dots,h$ one has that $\big<\mu_1,\dots,\mu_{k-1}\big>:_R\mu_k$ is an ideal generated by linear forms. It is easy to see that ideals with linear quotients have linear resolutions. 

We denote by $\Omega$  the set of the monomials $\mu$ such that $\deg(\mu)=\sum_{i=1}^kt_i$ and for all $i=1,\dots,t_1$ we have $\gamma_{i,c}(\mu)\geq \gamma_i(\tau)$ where $\tau=(t_1,\dots,t_k)$.\\
By Lemma \ref{prim8}, we have:
\begin{proposition}
(i) $\Omega$ is a system of generators of J.\\
(ii) Let $\mu$ be a monomial with a decomposition $\mu = \eta_1\cdots \eta_v$ where the $\eta_i$ are $c$-chains. Set $s=\deg(\eta_1),\dots,\deg(\eta_v).$ Then $\gamma_{i,c}(\mu)\geq \gamma_i(s)$ for all $i=1,\dots,t_1.$ 
\end{proposition}
We introduce a total order $\sigma$ on the monomials of $R$ as follows. Let $\mu, \eta$ be monomials of $R$ and $\mu=\mu_1\cdots \mu_k$ and $\eta=\eta_1\cdots \eta_h$ their $c$-decompositions. We set $\mu >_{\sigma} \eta$ if $\mu_j > \eta_j$ in the degree lexicographic order for the first index $j$ such that $\mu_j \not= \eta_j.$

The following result can be proved by modifying the argument given in \cite[proposition 6.2]{CH}, just replace "$1$-chain"  with "$c$-chain".  
\begin{theorem}
\label{linquo}
Let $$J=J_{t_1}J_{t_2}\cdots J_{t_k},$$ where $J_t=\big<x_{a_1}\cdots x_{a_t}\; :\; a_1,a_2,\dots,a_t \mbox{ is a c-chain}\big>.$ Then J has linear quotients.
\end{theorem}

In this case, all generators of $J$ have the same degree. This implies that $J$ has a linear resolution over $R$.
Moreover, we have a well-known inequality for Betti numbers: $\beta_{ij}(R/I)\leq \beta_{ij}(R/\ini(I))$. One concludes:
\begin{theorem}
Let $$I=I_{t_1}I_{t_2}\cdots I_{t_k},$$ where the ideals $I_t$ are generated by the $t$-minors of $X_t$. Then $I$ has a linear resolution.
\end{theorem} 
\section{Quasi-Sorted Form and Rees Algebra}
In \cite{C}, the author studied the Rees algebra of determinantal ideals in the Hankel case. In this section, we deal with a more general case. We start with the following definition. 
\begin{definition}
Let $a=a_1,\dots,a_s$ be a $c$-chain. We define $L(a)$ to be the union of closed intervals $[a_i-c,a_i]$ for all $i=2,\dots,s$.\\
Let $a=a_1,\dots,a_s$ and $b=b_1,\dots,b_r$ be two $c$-chains with $x_a>_{\tau} x_b$. The pair $(a,b)$ is called \textit{quasi-sorted} if  $a_i\leq b_i$ for all $i=1,\dots,r$ and either $b_i\leq a_{i+1}$ for all $i=1,\dots,r$ or $b_i\leq a_{i+1}$   for all $i=1,\dots,k-1$ and $b_k>a_k+1$ and $b_k,\dots,b_r\in L(a).$ If $b_i\leq a_{i+1}$ for all $i=1,\dots,r$ then $(a,b)$ is called \textit{sorted}.

More generally, let $a^{(1)},\dots,a^{(k)}$ be $c$-chains with $a^{(i)}=a^{(i)}_1,\dots,a^{(i)}_{n_i}$ such that $x_{a^{(i)}}>_{\tau}x_{a^{(i+1)}}.$ The set $(a^{(1)},\dots,a^{(k)})$ is called \textit{sorted} if $a^{(t)}_i\leq a^{(s)}_i$ for all $t\leq s$ and $a^{(t)}_i\leq a^{(t')}_{i+1}$ for all $t'<t$. The set $(a^{(1)},\dots,a^{(k)})$ is called \textit{quasi-sorted} if it is either sorted or $a^{(t)}_i\leq a^{(s)}_i$ for all $t\leq s$  and if $a^{(t)}_i> a^{(t')}_{i+1}$ for some $t'<t$ then $a^{(t)}_i,a^{(t)}_{i+1},\dots,a^{(t)}_{n_i}\in L(a^{(t')}).$ We call the table $A=(a^{(i)}_j)$ quasi-sorted form for short. 
\end{definition}
\begin{example} Let c=2.\\
(1) 
$
\left (
\begin{array}{cccc}
1 & 4 & 8 & 11\\
3 & 7\\

\end{array}
\right )$
 is sorted because we have a \textit{zigzag} $1<3<4<7<8<11$.\\
(2) 
$
\left (
\begin{array}{cccc}
1 & 4 & 7 & 10\\
3 & 8\\

\end{array}
\right )$
is not sorted because $8>7$. But it is quasi-sorted because we have a \textit{zigzag} $1<3<4<8$ and $8\in L(1,4,7,10)$.
\end{example}

\begin{remark}
\label{rmqs}
(i) The set $(a^{(1)},\dots,a^{(k)})$ is \textit{quasi-sorted} if and only if the pair $(a^{(i)},a^{(j)})$ is quasi-sorted for all $1\leq i<j\leq k.$ \\
(ii) If $(a^{(1)},\dots,a^{(k)})$ is \textit{quasi-sorted} and $n_i=n_j>n_h$ with $i<j<h$ then $a^{(h)}_{n_h}\leq a^{(i)}_{n_i}\leq a^{(j)}_{n_j}.$

\end{remark}
\begin{lemma}
Let $a^{(1)},\dots,a^{(k)}$ be $c$-chains with $a^{(i)}=a^{(i)}_1,\dots,a^{(i)}_{n_i}$ such that $x_{a^{(i)}}>_{\tau}x_{a^{(i+1)}}$. If $n_i-n_j \leq 1 $ for all $i<j$ then $a^{(1)},\dots,a^{(k)}$ is sorted.
\end{lemma}
\begin{proof} We only prove that the pair $(a^{(i)},a^{(j)})$ is sorted for all $1\leq i<j\leq k.$\\
Assume that $a^{(j)}_t>a^{(i)}_{t+1}$. By definition we have $a^{(j)}_t,a^{(j)}_{t+1},\dots, a^{(j)}_{n_j} \in L(a^{(i)})$.
Set the sequence $\alpha_t,\dots,\alpha_{n_i}\in [n_i]$ such that $a^{(j)}_u\in [a^{(i)}_{\alpha_u}-c,a^{(i)}_{\alpha_u}]$. Since 
$a^{(j)}_t<_ca^{(j)}_{t+1}<_c\cdots <_ca^{(j)}_{n_j}$, we have $\alpha_t< \alpha_{t+1}< \cdots < \alpha_{n_j}$. Because $a^{(j)}_t>a^{(i)}_{t+1}$, we get $t+1<\alpha_t< \alpha_{t+1}< \cdots < \alpha_{n_j}$. Hence $n_i>n_j+1$, a contradiction. 
\end{proof}

Let $a=a_1,\dots,a_s$ and $b=b_1,\dots,b_r$ be two $c$-chains with $\prod_{i\in a}x_i\geq\prod_{i\in b}x_i$, and let $\Omega$ be the set of $c$-chains. We consider the following element of the polynomial ring $\mathbb{K}$[$Y_a$: $a \in \Omega$]:\medskip\\
\noindent
1) Pl\"ucker-type relations:  
$$Y_aY_b-Y_{a\wedge b}Y_{a\vee b}$$ where
$$a\wedge b=(\min(a_1,b_1),\dots,\min(a_r,b_r),a_{r+1},\dots,a_s),$$
$$a\vee b=(\max(a_1,b_1),\dots,\max(a_r,b_r))$$ 
and $a_h<b_h$, $a_k>b_k$ for some $h$ and $k$.\medskip\\
\noindent
2) New-type relations:
$$Y_aY_b-Y_cY_d$$
with $a_i\leq b_i$ for all $i=1,\dots,r$, and there exist  $1\leq h\leq k\leq r$ with 
$$
\begin{array}{l}
b_{h-1}\leq a_h, \\
b_h>a_{h+\alpha}>a_{h+1}, b_{h+1}>a_{h+\alpha+1}, \dots,
b_k>a_{k+\alpha}, \\
b_{k+1}<_c a_{k+\alpha +1},
\end{array}
$$
 where
$$c=(a_1,\dots,a_{h+\alpha-1},b_h,b_{h+1},\dots,b_k,a_{k+\alpha+1},\dots,a_s)$$ and 
$$d=(b_1,\dots,b_{h-1},a_{h+\alpha},a_{h+\alpha+1},\dots,a_{k+\alpha},b_{k+1},\dots,b_r).$$
\medskip
  
By a \textit{ marked polynomial } we mean a polynomial $f\in R-\{0\}$ together with a specified term \textit{in(f)} of $f$. Here $in(f)$ can be any term appearing in $f$. Given a set $\mathcal{F}$ of marked polynomials,  we define the \textit{ reduction relation modulo  $\mathcal{F}$} in the usual sense of Gr\"obner bases. We say that $\mathcal{F}$ is \textit{ marked coherently } if there exists a term order $\prec$ on $R$ such that $in(f)=in(f)$ for all $f$ in $\mathcal{F}$. Clearly, if $\mathcal{F}$ is marked coherently, then the reduction relation $"\rightarrow_{\mathcal{F}}"$ is Noetherian. In  \cite[Theorem 3.12]{St} we have:
\begin{theorem}
\label{thmqs1}
A finite set $\mathcal{F}\subset R$ of marked polynomials is marked coherently if and only if the reduction relation modulo $\mathcal{F}$ is Noetherian, i.e., every sequence of reductions modulo $\mathcal{F}$ terminates.  
\end{theorem}
 In this case, we have a set of marked polynomials 
\begin{eqnarray*}
G =&\{& \underline{Y_aY_b}-Y_{a\wedge b}Y_{a\vee b}, \mbox{ in Pl\"ucker-type relations }\\
& & \underline{Y_aY_b}-Y_cY_d, \mbox{ in New-type relations} \}.
\end{eqnarray*}
\begin{lemma}
\label{qs1}
Let $a=a_1,\dots,a_s$ and $b=b_1,\dots,b_r$ be two $c$-chains with $\prod_{i\in a}x_i\geq\prod_{i\in b}x_i$. The pair $(a,b)$ always reduces modulo $G$ to a quasi-sorted pair of the same size of $(a,b)$.
\end{lemma}
\begin{proof}
By using Pl\"ucker-type relations we may always assume that $a_i\leq b_i$ for all $i=1,\dots,r$. 
If the pair $(a,b)$ is not quasi-sorted then we have that there exists a pair $(h,k)$ such that $b_i\leq a_{i+1}$ for all $i=1,\dots,h-1$, $b_h>a_{h+1}$, $b_h,\dots,b_{k-1}\in L(a)$ and $b_k\notin L(a).$ Because $b_h,\dots,b_{k-1}\in L(a)$, $b_k\notin L(a)$ we can assume that $b_i\in [a_{t_i}-c,a_{t_i}]$ for all $i=h,\dots,k-1$ and $b_k\in(a_{t_k},a_{t_k+1}-c)$. Since $b_h<_cb_{h+1}<_c\cdots<_cb_{k-1}<_cb_k$, we have $t_h<t_{h+1}<\cdots<t_{k-1}\leq t_k.$ \\
First, if $t_{k-1}<t_k$ then we can replace $b_k$ by $a_{t_k}$ using New-type relations. After a finite number of steps we reduce to the case $t_{k-1}=t_k$.\\
Second, if $t_{k-1}>t_{k-2}+1$ then we can replace $b_{k-1},b_k$ by $a_{t_k-1},a_{t_k}$ using New-type relations. After a finite number of steps we reduce to the case $t_{k-1}=t_{k-2}+1$. Proceeding in this way, we obtain $t_{k}=t_{k-1}$, $t_i=t_{i-1}+1$ for all $i=2,\dots,k-1.$
We replace $b_h,\dots,b_k$ by $a_{t_h-1},a_{t_h},\dots,a_{t_{k-1}}$ using New-type relations. 
By induction on $(h,k)$ we can reduce the pair $(a,b)$ modulo $G$ to a quasi-sorted pair.  
\end{proof}
\begin{corollary}
\label{qs2}
Let $a^{(1)},\dots,a^{(k)}$ be $c$-chains with $a^{(i)}=a^{(i)}_1,\dots,a^{(i)}_{n_i}$ such that $\prod_{i\in a^{(i)}}x_i\geq\prod_{i\in a^{(i+1)}}x_i$. The table $A=(a^{(i)}_j)$ always reduces modulo $G$ to a quasi-sorted form of the same size of $A$.
\end{corollary}
\begin{proof}
By using Pl\"ucker-type relations and New-type relations we can assume that the table $A=(a^{(i)}_j)$ for $i=1,\dots,k$, $j=1,\dots,n_i$ with the columns increase from top to bottom and the rows are $c$-chains . An entry $a^{(i)}_j$ is called a normal entry if it satisfies that either $a^{(i)}_j\leq a^{(i)}_{j+1}$ or $a^{(i)}_j> a^{(i')}_{j+1}$ and $a^{(i)}_h\in L(a^{(i')})$ for all $h=j,\dots,n_i$.\\
Obviously, table $A$ is a quasi-sorted form if and only if $a^{(i)}_j$ is a normal entry for all $i=1,\dots,k$, $j=1,\dots,n_i$. Applying step by step Lemma \ref{qs1} we have that $a^{(i)}_j$ is a normal entry. So the table $(a^{(i)}_j)$ always reduces modulo $G$ to a quasi-sorted form. 
\end{proof}
\begin{lemma}
\label{qs3}
Let $a^{(1)},\dots,a^{(k)}$ be $c$-chains with $a^{(i)}=a^{(i)}_1,\dots,a^{(i)}_{n_i}$ such that $\prod_{i\in a^{(i)}}x_i\geq\prod_{i\in a^{(i+1)}}x_i$. If the table $A=(a^{(i)}_j)$ reduces to quasi-sorted form $B=(b^{(i)}_j)$ of the same size of $A$, then $B$ is unique.
\end{lemma}
To prove this lemma we need to label the entries of the table of the same size as $A$ by the following algorithm:
\begin{algorithm}
Set $n_{k+1}:=0$ and $PF(0,1):=0$. \\
For $t=k$ Down To 1 Do\\
If $n_{t+1}=n_t$ Then $t:=t+1$ Else\\
For $i=n_{t+1}+1$ To $n_t$ Do\\
For $j=1$ To $t$ Do $PF(i,j):=PF(t,n_t)+j+(i-1)t$
\end{algorithm}
\begin{example}
If $n_1=n_2=7,n_3=n_4=4,n_5=n_6=2$ we have the labeling:
$$
PF=
\begin{array}{ccccccc}
1 & 7 & 13 & 17 & 21 & 23 & 25\\
2 & 8 & 14 & 18 & 22 & 24 & 26\\
3 & 9 & 15 & 19 \\
4 & 10& 16 & 20 \\
5 & 11\\
6 & 12
\end{array}
$$
\end{example}
\begin{proof}[Proof of \ref{qs3}]
We have that the above function $PF(i,j)$ accepts the values $1,\dots,l$ with $l=\sum_1^kn_i$.
We order the multiset $A=(a^{(i)}_j)$ (the same for $B$) by the multiset $\{c_1,c_2,\dots,c_l\}$, namely $c_1\leq c_2 \leq \cdots \leq c_l$. We have the unique property of the quasi-sorted form $B$ given by the place of $c_t$ in the form $B$.\\
We prove by decreasing induction on $t$. \\
If $t=l$ then there exists a place $(i,j)$ such that $PF(i,j)=l$. Using the second part of Remark \ref{rmqs} we have $b^{(i)}_j=c_l$ and we replace the $PF$-function by setting $PF(i,j)=0$.\\
Assume that we defined the place $c_h=b^{(\alpha)}_{\beta}$ and $PF(\alpha,\beta)=0$ for all $h>t$. We restart with $(i,j)$, where $i$ is the row index and $j$ is the column index in $B$, such that $PF(i,j)$ is maximal. By the definition of quasi-sorted form, if $j<n_i$ and $c_t<_cb^{(i)}_{j+1}$ then $b^{(i)}_j=c_t$ otherwise we continue with the next largest value $PF(i,j)$. Hence, the place of $c_t$ in $B$ is defined.  
\end{proof} 

Let $I$ be an ideal of a ring $R$. The Rees algebra $\Rees(I)$ of $I$ is the $R$-graded algebra $\bigoplus_{k=0}^\infty \ I^{k}T^k$, where $T$ is an indeterminate over $R$.  In other words, $\Rees(I)$   can be identified with the $R$-subalgebra of  $R[T]$ generated by $IT$. We may also consider the symbolic Rees algebra $\Rees^s(I)$, that is,
$\Rees^s(I)=\bigoplus_{k=0}^\infty I^{(k)}T^k$. If $R$  is a polynomial ring and
$\tau$ a monomial order,  then the initial algebra of $\Rees(I)$ is  $\ini_\tau(\Rees(I))=\bigoplus_{k=0}^\infty \ini_\tau( I^k)T^k$. Similarly the initial algebra of $\ini(\Rees^s(I))$ of $\Rees^s(I)$ is $\ini_\tau(\Rees^s(I))=\bigoplus_{k=0}^\infty \ini_\tau( I^{(k)})T^k$.\\ 
\begin{proposition} 
\label{ree1}
One has: 
$$
\begin{array}{l}
\Rees^s(I_t)=\mathbb{K}[x][I_tT,I_{t+1}T^2,\dots,I_mT^{m-t+1}]\\ \\ 
\ini(\Rees^s(I_t))=\mathbb{K}[x][\ini(I_t)T,\ini(I_{t+1})T^2,\dots,\ini(I_m)T^{m-t+1}].
\end{array}
$$
In particular, $\Rees^s(I_t)$ and  $\ini(\Rees^s(I_t))$ are Noetherian, Cohen-Macaulay  normal domains.
\end{proposition}
For the proof of Proposition \ref{ree1}, one uses exactly the same arguments given by Conca in \cite{C}.

Let $I_1,\dots,I_s$ be ideals of a ring $R$. The multi-homogeneous Rees algebra $\Rees(I_1,\dots,I_s)$ of $I_1,\dots,I_s$ is the $R$-graded algebra $$\Rees(I_1,\dots,I_s)=\bigoplus_{\alpha_1,\dots,\alpha_s} \ (I_1T_1)^{\alpha_1}\cdots (I_sT_s)^{\alpha_s}, $$ where $T_1,\dots,T_s$ are  indeterminates over $R$.

Let $I_{t_1},\dots,I_{t_k}$ be determinantal ideals of extended Hankel  matrices. We have $\ini (\Rees(I_{t_1},\dots,I_{t_k}))=\Rees(J_{t_1},\dots,J_{t_k})$. By Corollary \ref{prim10} we have the following result:
\begin{proposition}  The multi-homogeneous Rees algebra $\Rees(I_{t_1},\dots,I_{t_k})$ is normal and Cohen-Macaulay.   
\end{proposition} 
In \cite{C} and \cite{CHV}, the authors studied the presentation of the Rees algebras for $s=1$. In this part we would like to treat the more general case:
\begin{theorem}
\label{ree}
The multi-homogeneous Rees algebra $\Rees(I_{t_1},\dots,I_{t_k})$ is defined by 
 a  Gr\"obner basis of quadrics.
\end{theorem}
By virtue of \cite[Corollary 2.2]{CHV}, it suffices to show that the initial algebra of $\Rees(I_{t_1},\dots,I_{t_k})$ is defined by a Gr\"obner basis of quadrics. In this case the initial algebra is $\Rees\big(J_{t_1},\dots,J_{t_k}\big)$.

Let $A=\{(i,a_1,\dots,a_{t_i}): i=1,\dots k, a_1<_ca_2<_c \cdots <_c a_{t_i}\}$ and take a family  of indeterminates  $Y=(Y_a)_{a\in A}$.\\
Consider the presentation of
$\Rees\big(J_{t_1},\dots,J_{t_k}\big)$ 
$$ \Phi:\mathbb{K}[x][Y]\to \Rees\big(J_{t_1},\dots,J_{t_k}\big)$$ 
 is defined by sending $x_i$ to $x_i$ and $Y_a$ to $x_aT_j=x_{a_1}x_{a_2}\cdots x_{a_{t_j}}T_j$, where $a=(j,a_1,\dots,a_{t_j})$.\\
In particular, the presentation of the special fiber of $\Rees(J_{t_1},\dots,J_{t_k})$
$$ \Psi:\mathbb{K}[Y]\to \Rees(J_{t_1},\dots,J_{t_k})/m_R\Rees(J_{t_1},\dots,J_{t_k})$$ defined by sending $Y_a$ to $x_aT_j=x_{a_1}x_{a_2}\cdots x_{a_{t_j}}T_j$.\\
The defining ideal of the special fiber of the multi-homogeneous Rees algebra $\Rees(J_{t_1},\dots,J_{t_k})$ is
\begin{eqnarray*}
I_{\Psi}=&\big<&Y_{(i_1,a^{(1)})}\cdots Y_{(i_k,a^{(k)})}-Y_{(j_1,b^{(1)})}\cdots Y_{(j_k,b^{(k)})}:\; i_p=j_p \forall p,\\
& & \multiset(a^{(1)}\cup \cdots \cup a^{(k)})=\multiset(b^{(1)}\cup \cdots \cup b^{(k)})\big>,
\end{eqnarray*}
where $a^{(i)},b^{(j)}$ are $c$-chains.\\
We will show that $I_{\Psi}$ is defined by  a  Gr\"obner basis of quadrics.

In the polynomial ring $\mathbb{K}[Y_{(i,a)}: a \mbox{ is c-chain length }t_i]$, a monomial\\ $Y_{(i_1,a^{(1)})}\cdots Y_{(i_k,a^{(k)})}$ is called quasi-sorted if $(a^{(1)},\dots,a^{(k)})$ is quasi-sorted. We have the following theorem:
\begin{theorem}
\label{quadratic}
There exists a term order $\prec$ on $\mathbb{K}[Y]$ such that the quasi-sorted monomials are precisely the $\prec$-standard monomials modulo $I_{\Psi}$. The initial ideal $\ini(I_{\Psi})$ is generated by square-free quadratic monomials.\\
In particular, the special fiber $\Rees(J_{t_1},\dots,J_{t_k})/m_R\Rees(J_{t_1},\dots,J_{t_k})$ is defined by  a  Gr\"obner basis of quadrics. 
\end{theorem}
\begin{proof}
Let $G_0$ denote the set of marked binomials 
$$\{ \underline{Y_{(s,a)}Y_{(r,b)}}-Y_{(s,c)}Y_{(r,d)}:\mbox{(c,d) is the quasi-sorted pair reduction from (a,b)}\}.$$
Obviously, these relations do indeed lie in $I_{\Psi}$. Since Corollary \ref{qs2} shows that the reduction relation defined by $G_0$ is Noetherian, by Theorem \ref{thmqs1}, this implies that there exists a term order $\prec$ on $\mathbb{K}[Y]$ which selects the underlined term as the initial term for each binomial in $G_0$.\\
Consider the initial ideal $\ini(I_{\Psi})$. Every monomial which is not quasi-sorted lies in this ideal. Assume that some quasi-sorted monomial $m_1$ lies in $\ini(I_{\Psi})$. There exists a non-zero binomial $m_1-m_2\in I_{\Psi}$ such that $m_2$ does not lie in $\ini(I_{\Psi})$. So $m_2$ is a quasi-sorted monomial. This implies $m_1,m_2$ are quasi-sorted monomials which lie in the same residue class modulo $I_{\Psi}$. By Lemma \ref{qs3} we have $m_1=m_2$. This is a contradiction. Hence the monomials in $\ini(I_{\Psi})$ are precisely the non-quasi-sorted monomials. We conclude that the set $G_0$ is a  Gr\"obner basis of $I_{\Psi}$ with respect to $\prec$.
\end{proof}
Moreover by setting $Y_{(0,t)}=x_t$, we have that the Rees algebra $\Rees(J_{t_1},\dots,J_{t_k})$ is defined by  a  Gr\"obner basis of quadrics of following forms:\\
(i)  $Y_{(s,a)}Y_{(r,b)}-Y_{(s,c)}Y_{(r,d)}$:  (c,d) is the quasi-sorted pair reduction of (a,b).\\
(ii)  $x_tY_{(p,a)}-x_{a_h}Y_{(p,b)}$ : with  $a_{h-1}<_c t<a_h$ for some $h$, $1\leq h \leq p$, $b$ is the
sequence $(a_1,\dots,a_{h-1},t,a_{h+1},\dots,a_p)$ and $a_0=-\infty$. Hence we have proved Theorem \ref{ree}.

Moreover, we can deduce that the multi-homogeneous Rees algebra\\ $\Rees(I_1,\dots,I_s)$ is Koszul; see \cite[Corollary 3.14]{BC2}. By using the result in \cite{B} for the multigraded case, we can give another proof of the result in Section 3.

\vspace{1cm}

\end{document}